\documentclass[a4paper]{amsart}

\usepackage{enumerate}
\usepackage{amssymb}
\usepackage{amsthm,amsmath,amsfonts}
\usepackage{hyperref}
\usepackage{color}
\usepackage{graphicx}

\newcommand\eps{{\epsilon}}

\newcommand\grad{{\bf \nabla}}

\newcommand\nvec{{\bf n}}
\newcommand\xvec{{\bf x}}
\newcommand\uvec{{\bf u}}

\newcommand\mvec{{\bf m}}
\newcommand\pvec{{\bf p}}
\newcommand{\xbar}{\bar{\mathbf{x}}}
\newcommand\xhat{\hat{\bf x}}

\newcommand\Evec{{\bf E}}
\newcommand\Fvec{{\bf F}}
\newcommand\Gvec{{\bf G}}
\newcommand\Xvec{{\bf X}}
\newcommand\Yvec{{\bf Y}}

\newcommand{\Wvec}{\mathbf{W}}
\newcommand\Qvec{{\bf Q}}
\newcommand\Vvec{{\bf V}}
\newcommand\Hvec{{\bf H}}

\newcommand\Vbar{\bar{\bf V}}
\newcommand{\Lbar}{\bar{L}}
\newcommand{\Qbar}{\bar{\mathbf{Q}}}

\newcommand{\Rr}{{\mathbb R}}

\newcommand{\N}{\mathbb{N}}
\renewcommand{\d}{\mathrm{d}}

\DeclareMathOperator{\tr}{tr}

\newcommand\norm[1]{\|#1\|}
\newcommand{\abs}[1]{\left|#1\right|}

\newcommand{\der}[1]{{#1}^{\prime}}
\newcommand{\sder}[1]{{#1}^{\prime\prime}}

\newtheorem{theor}{Theorem}
\newtheorem{lemma}[theor]{Lemma}
\newtheorem{proposition}[theor]{Proposition}
\theoremstyle{definition}
\newtheorem{remark}{Remark}[section]
\newtheorem{step}{Step}

\title{Radial Symmetry on Three-dimensional Shells in the Landau-de Gennes Theory}

\author{Giacomo Canevari}
\thanks{GC's present address is: Mathematical Institute, University of Oxford, Andrew Wiles Building,
 Radcliffe Observatory Quarter, Woodstock Road, Oxford~OX2~6GG, United Kingdom.}
\address[G. Canevari]{Sorbonne Universit\'es, UPMC Univ Paris 06, CNRS, UMR 7598,
Laboratoire Jacques-Louis Lions, 4~place~Jussieu, 75005~Paris, France.}
\email{canevari@maths.ox.ac.uk}

\author{Mythily Ramaswamy}
\address[M. Ramaswamy]{Tata Institute of Fundamental Research, Centre for Applicable Mathematics, 
Sharada Nagar, Chikkabommasandra, Bangalore~560065, India.}
\email{mythily@math.tifrbng.res.in}

\author{Apala Majumdar}
\thanks{AM is the \emph{corresponding author}.
 AM's research is also supported by an EPSRC Career Acceleration Fellowship EP/J001686/1,
 an OCIAM Visiting Fellowship and the Keble Advanced Studies Centre.}
\address[A. Majumdar, corresponding author]{Department of Mathematical Sciences, University of Bath, Claverton Down, Bath~BA2~7AY, United~Kingdom.}
\email{a.majumdar@bath.ac.uk}

\begin{document}

\begin{abstract}
We study the
radial-hedgehog solution on a three-di\-men\-sio\-nal (3D)
spherical shell with radial boundary conditions, within the
Landau-de Gennes theory for nematic liquid crystals. 
We prove that the radial-hedgehog solution is the unique minimizer of the Landau-de Gennes energy in two separate regimes:
(i) for thin shells when the temperature is below the critical nematic supercooling temperature and 
(ii) for a fixed shell width at sufficiently low temperatures. In case (i), we provide explicit geometry-dependent criteria 
for the global minimality of the radial-hedgehog solution.
\end{abstract}


\maketitle

\section{Introduction}
Nematic liquid crystals are anisotropic liquids with long-range
orientational ordering i.e. the constituent molecules have no translational order but exhibit directional order 
in the sense that they tend to align along certain distinguished directions~\cite{dg, newtonmottram}. 
Nematic liquid crystals have generated tremendous academic interest in recent years,
partly for fundamental scientific reasons and partly for their widespread applications in materials science and nano-technology~\cite{lagerwall}. 
Defects in liquid crystals fascinate mathematicians and applied scientists alike and there has been substantial 
recent analytical work on defects, following the seminal work of Schopohl and Sluckin in~\cite{sluckin}. 
Whilst defects pose numerous mathematical and applications-oriented challenges for liquid crystal research, 
it is also important to rigorously analyze defect-free configurations and in particular, rigorously characterize model situations
where we expect to see stable defect-free nematic configurations and if such defect-free states can be exploited for new applications. 

Continuum theories for nematics e.g.
Oseen-Frank, Ericksen and Landau-de Gennes theories, of which the Landau-de Gennes
theory is the most general, have received
considerable attention in the mathematical and modelling literature
\cite{partialcrystal, ericksen, lin}.
The radial-hedgehog solution is the classical example of a point
defect on a three-dimensional spherical droplet, in the Landau-de Gennes theory~\cite{dg,virga}.
The radial-hedgehog solution has a relatively straightforward structure: the molecules point radially outwards everywhere
away from the droplet centre, with a disordered ``isotropic'' defect core located at the centre.
The radial-hedgehog solution has received substantial mathematical interest in recent 
years~\cite{am2, sima2012, lamy, ignat2014, canevari2014}. This is, to some extent, because the radial-hedgehog solution
is a relatively rare example of an explicit critical point of the Landau-de Gennes energy functional and therefore, 
naturally more amenable to analytical methods.
Further, the radial-hedgehog solution is analogous to the degree $+1$-vortex in the Ginzburg-Landau theory for superconductivity~\cite{bbh}.
The degree $+1$-vortex is a well studied solution in the Ginzburg-Landau community~\cite{bbh, millot, pisante}. 
This means that we can borrow several ideas and methods from Ginzburg-Landau theory to address non-trivial questions
about the structure and stability of the radial-hedgehog solution. However, there is a crucial difference
between Ginzburg-Landau theory and Landau-de Gennes theory. In the Ginzburg-Landau framework, 
we typically deal with three-dimensional vectors on $\Rr^3$ i.e. maps $\uvec\colon \Rr^3 \to \Rr^3$, or more generally, 
$N$-dimensional vectors defined on~$\Rr^N$. When we work with the Landau-de Gennes theory, 
we study a nonlinear coupled system of partial differential equations for a five-dimensional tensor-valued
$\Qvec$-order parameter defined on a three-dimensional domain i.e.
we study maps, $\Qvec\colon \Omega \subset \Rr^3 \to \Rr^5$. There
are two additional degrees of freedom which can drastically alter
the solution landscape in spite of apparent mathematical
similarities between the Landau-de Gennes system and the
Ginzburg-Landau system~\cite{amaz, sima2012}.

For example, in~\cite{millot, pisante},
the authors prove that the degree $+1$-vortex solution is the
unique solution (up to translation and rotation) of the
Ginzburg-Landau equations on $\Rr^3$, subject to certain natural
energy bounds and topologically non-trivial boundary conditions.
However, it is
known that the analogous radial-hedgehog solution loses stability in the Landau-de Gennes framework, on a three-dimensional droplet with radial boundary
conditions, for sufficiently low temperatures, see~\cite{sonnet, mkaddem&gartland2, sima2012, am2}.
The geometry and the boundary conditions enforce the
radial-hedgehog solution to have an ``isotropic'' core at the
droplet center and the isotropic core is energetically expensive
for low temperatures, leading to the instability with respect to higher-dimensional perturbations. 

In this paper, we focus on radial equilibria on a 3D spherical shell, bounded by two spherical surfaces with 
radial boundary conditions on both the inner and outer spherical
surfaces. Nematics in spherical shells have received much  interest since Nelson's seminal work in 2002~\cite{nelson}. 
Since then, it has been widely recognized that nematics in shells offer ample scope for generating novel non-singular and
singular nematic textures and these textures can be controlled by shell thickness, shell heterogeneity 
(concentric versus non-concentric shells), temperature and material elastic constants~\cite{ravnik1, ravnik2}. 
In some cases, these textures in 3D shells naturally exhibit defects and these defects can act as binding sites or
functionalization sites, leading to new material possibilities~\cite{ravnik1, ravnik2}.

We, firstly, prove the existence of a radial-hedgehog type solution on a 3D spherical shell with radial boundary conditions 
i.e. an explicit critical point of the Landau-de Gennes energy with perfect radial symmetry. 
In Section~\ref{sec:RH}, we provide an analytical description of this radial-hedgehog solution, by analogy with similar work 
on a 3D droplet. The radial-hedgehog solution is defined by a scalar order
parameter, $h$, which vanishes at isotropic points or defect points
\cite{sima2012,am2}. We show that the radial-hedgehog
solution has no isotropic/zero points on a 3D spherical shell,
for all temperatures below the nematic supercooling temperature. 
In other words, the radial-hedgehog solution is a defect-free equilibrium for this model problem. 
For a concentric shell without external fields, as considered in our manuscript, the stability of the radial-hedgehog
solution is controlled by temperature, shell width and material elastic constants. In the limit of vanishing
elastic constants, one can prove that minimizers of a relatively
simple Landau-de Gennes energy converge uniformly to the
radial-hedgehog solution on a 3D spherical shell with Dirichlet radial boundary conditions,
by appealing to the results in~\cite{amaz}. We work with fixed elastic constants and instead focus on the interplay between temperature and shell width in this paper.

The two key theorems in this paper are stated below. In Section~\ref{sec:global}, we focus on narrow shells
with no restriction on the temperature $t$, except for that $t \geq 0$ so that we are working with temperatures 
below the critical nematic supercooling temperature.
\begin{theor} \label{th:1}
Let $\Omega := \left\{ \xvec \in \Rr^3\colon 1 \leq |\xvec| \leq R \right\}$ and
\[
 R < \min\left\{R_0 := \exp\left( \frac{4\pi^2}{23}\right), \, R^* \right\} 
\]
where $R^*$ is defined in Proposition~\ref{prop:1}. Then the
radial-hedgehog solution is the unique global minimizer of the
Landau-de Gennes problem~\eqref{pb:Landau-dG} in the admissible
class $\mathcal{A}$ defined in~\eqref{eq:admissible} (see Section~\ref{sec:prelim}), for all
temperatures below the critical nematic supercooling temperature.
\end{theor}
In~\cite{golovaty}, Golovaty and Berlyand prove uniqueness and radial symmetry of the minimizer on a thin 2D annulus, in the Ginzburg-Landau theory.
Theorem~\ref{th:1} is an analogous result for a thin 3D shell in the Landau-de Gennes theory, with a method of proof based on the Landau-de Gennes energy itself.

In Section~\ref{sec:t}, we study the effect of the reduced
temperature, $t$, on the stability of the radial-hedgehog
solution.
Our second main result demonstrates
the global minimality of the radial-hedgehog solution in the $t\to \infty$ limit.
\begin{theor} \label{th:2}
Let $\Omega \subset \Rr^3$ be a 3D spherical shell as defined
above. There exists $\tau \geq 0$ such that, for any~$R > 1$ and any temperature $t\geq \tau$, the radial-hedgehog
is the unique global minimizer for Problem~\eqref{pb:Landau-dG}.
\end{theor}
Our mathematical strategy is similar for Theorems~\ref{th:1} and~\ref{th:2}. In Theorem~\ref{th:1},
we compute an explicit sub-solution for the order parameter, $h$, that only depends on the shell width
and is independent of $t$. In particular, we can use the shell width to uniformly control the magnitude of $h$,
a property which is absent for spherical droplets. We prove the global minimality of the radial-hedgehog solution
by writing the energy of an arbitrary nematic state (in the admissible class $\mathcal{A}$ defined in~\eqref{eq:admissible})
as the sum of the second variation of the Landau-de Gennes energy about the radial-hedgehog solution and the higher-order 
cubic and quartic contributions.
We control the second variation by means of a Poincar\'e-type inequality, purely in terms of the shell width, 
and use algebraic methods to prove the non-negativity of the residual terms.

In Theorem~\ref{th:2}, we show that the temperature $t$ uniformly controls the magnitude of $h$, for all fixed shell widths
independent of $t$. We again write the energy of an arbitrary nematic state as the sum of the second variation 
of the Landau-de Gennes energy about the radial-hedgehog solution and the higher-order cubic and quartic contributions. 
We control the second variation by adapting arguments in~\cite{ignat2014}; in particular, we derive an explicit positive lower bound
for the second variation which gives us greater control on the residual cubic and quartic energy terms. 
In particular, the sum of the cubic and quartic energy terms can be negative for large $t$, 
so global minimality is not guaranteed by non-negativity of the second variation alone.
The improved lower bound for the second variation allows us to control the problematic (potentially negative) terms
in the energy expansion for sufficiently large $t$, leading to the desired conclusion above.

The radial-hedgehog solution is a defect-free radial equilibrium for this model problem. 
A rigorous analysis of defect-free equilibria is the first step in the analysis of generic nematic equilibria in shells
and from an applications perspective, radial equilibria can also act as binding sites or attractors for microparticles
with compatible boundary conditions, leading to new material possibilities. 
We prove our results with a Dirichlet radial boundary condition.
However, we expect them to be true with surface anchoring potentials too, for sufficiently large values of the anchoring strength.
Further, the global minimality of the radial-hedgehog solution on a shell may seem intuitive to some readers. 
Whilst the passage from physical intuition to mathematical proof is always worthwhile, it is important to point out 
that global minimizers of the Ginzburg-Landau energy on 2D annuli, with fixed topological degree on the boundary 
(compatible with the radial-hedgehog solution in 2D), lose radial symmetry and develop vortices/defects
for thick annuli or large annulus width, in the $\eps \to 0$ limit~\cite{berlyandvoss}. 
The $\eps \to 0$ limit mimics, to some extent, the $t \to \infty$ limit in the Landau-de Gennes theory. 
Therefore, a rigorous proof of the global minimality of the defect-free radial-hedgehog solution on a 3D shell,
in the Landau-de Gennes framework, in the $t\to\infty$ limit, excludes such possibilities.

\numberwithin{equation}{section}
\numberwithin{theor}{section}

\section{Preliminaries}
\label{sec:prelim}

We work within the Landau-de Gennes theory for nematic liquid
crystals wherein the nematic configuration is described by the
$\Qvec$-tensor order parameter~\cite{dg}. Mathematically, the $\Qvec$-tensor
is a symmetric, traceless $3\times 3$
matrix. Let $S_0$ denote the space of all symmetric, traceless
$3\times 3$ matrices defined by
\begin{equation*} \label{eq:S0}
S_0 := \left\{ \Qvec \in \mathbb{M}^{3\times3}\colon \sum_{i = 1}^3 \Qvec_{ii} = 0 \textrm{ and } \Qvec_{ij}=\Qvec_{ji}
\textrm{ for } i,j=1,2,3 \right\}.
\end{equation*}
The domain is a 3D spherical shell, with outer radius $R$ and inner radius set to unity, as shown below
\begin{equation*}
\label{eq:omega}
\Omega := \left\{ \xvec \in \Rr^3\colon 1 \leq |\xvec| \leq R \right\} \quad \textrm{where} \quad R>1.
\end{equation*}
A $\Qvec$-tensor is said to be (i) isotropic when $\Qvec=0$, (ii) uniaxial when $\Qvec$ has two degenerate non-zero eigenvalues 
and (iii) biaxial when $\Qvec$ has three distinct eigenvalues~\cite{dg,virga}. 
A uniaxial~$\Qvec$-tensor can be written in the form 
\begin{equation*}
\label{eq:uniaxial}
\Qvec_u = s \left(\nvec
\otimes \nvec - \frac{\mathbf{I}}{3} \right)
\end{equation*}
for a real-valued order parameter, $s$, and a unit-vector field $\nvec
\in S^2$ i.e. $\Qvec_u$ has three degrees of freedom whereas a
biaxial $\Qvec$-tensor uses all five degrees of freedom. In
physical terms, a uniaxial $\Qvec$-tensor corresponds to a nematic
configuration with a single distinguished direction of molecular
alignment whereas a biaxial $\Qvec$-tensor corresponds to a
configuration with two preferred directions of molecular
alignment.

We consider a simple form of the Landau-de Gennes energy given by~\cite{dg,newtonmottram}
\begin{equation*}
\label{eq:1}
I[\Qvec] := \int_{\Omega} \frac{L}{2}|\grad \Qvec|^2 + f_B\left(\Qvec\right)~\d V.
\end{equation*}
 In what follows, we assume that the elastic constant $L > 0$ is fixed once and for all, 
 since the $L\to 0$ limit has been well-studied in recent years~\cite{amaz}.
 We use Einstein summation convention throughout the paper i.e.
$\left| \grad \Qvec \right|^2 = \Qvec_{ij,k}\Qvec_{ij,k}$ and $i,j,k=1,2,3$.
The bulk potential, $f_B$, drives the nematic-isotropic phase transition and for the purposes of this paper, 
we take $f_B$ to be a quartic polynomial in the $\Qvec$-tensor invariants as shown below:
\begin{equation*}
\label{eq:2}
f_B(\Qvec) := \frac{A}{2}\tr\Qvec^2 -
\frac{B}{3}\tr\Qvec^3 +
\frac{C}{4}\left(\tr\Qvec^2 \right)^2,
\end{equation*}
$\tr\Qvec^2 = \Qvec_{ij}\Qvec_{ij}$,
$\tr\Qvec^3= \Qvec_{ij}\Qvec_{jp}\Qvec_{pi}$ and
$i,j,p=1,2,3$. We have $A = \alpha (T - T^*)$, where
$\alpha>0$ is a material-dependent constant, $T$ is the
temperature and $T^*$ is the critical nematic supercooling
temperature~\cite{ejam, newtonmottram}. We work with temperatures
$T\leq T^*$, so that $A\leq 0$, and we treat $B, \, C>0$ to be fixed
material-dependent constants.

For $A\leq 0$, a standard computation shows that $f_B$ attains its minimum
on the set of \emph{uniaxial} $\Qvec$-tensors given by~\cite{ejam}
\begin{equation}
\label{eq:3} \Qvec_{\min} := \left\{ \Qvec \in S_0 \colon
\Qvec=s_+\left(\nvec\otimes \nvec - \frac{\mathbf{I}}{3} \right)
\right\},
\end{equation} $\nvec\in S^2$ is an arbitrary unit vector and
\begin{equation*}
\label{eq:s+}
 s_+ := \frac{B + \sqrt{B^2 + 24 |A| C}}{4 C}.
 \end{equation*}
We introduce the scalings
\begin{alignat*}{3} \label{eq:4}
t     &:= \frac{27 |A| C}{B^2}; \qquad && h_+   := \frac{3 +\sqrt{9 + 8t}}{4} \\
\Lbar &:= \frac{27 C L}{2 B^2}; \qquad && \Qbar \hspace{4pt}:= \frac{1}{s_+}\sqrt{\frac{3}{2}}\Qvec \\
\xbar &:= \bar{L}^{-1/2}\xvec .
\end{alignat*}
One can easily verify that
\begin{equation*}
\label{eq:5} 
s_+ = \frac{B}{3C} h_+; \quad 2h_+^2 = 3h_+ + t.
\end{equation*}
In what follows, we refer to $t$ as the reduced temperature and always work with $t\geq 0$. The re-scaled domain is
\begin{equation}
\label{eq:omegabar}
\bar{\Omega}:=\left\{\xvec\in\Rr^3 \colon \bar{L}^{-1/2} \leq |\xbar| \leq \bar{L}^{-1/2}R \right\}.
\end{equation}
We measure the dimensionless length in units of $\bar{L}^{1/2}$, hence we can assume WLOG that~$\bar{L} = 1$ and~\eqref{eq:omegabar} is equivalent to
\begin{equation*}
\label{eq:omegabar2}
\bar{\Omega}=\left\{\xvec\in\Rr^3 \colon 1 \leq |\xbar| \leq R \right\}
\end{equation*}
where $R>1$ is the dimensionless outer radius.

We drop the \emph{bars} in what follows and all
statements are to be understood in terms of the re-scaled
variables. The re-scaled Landau-de Gennes functional is given by
\begin{equation}
\label{eq:6} I[\Qvec] = \int_{\Omega}\frac{1}{2}|\grad \Qvec|^2 + \frac{t}{8}\left[
\left(1 - |\Qvec|^2 \right)^2 + \frac{h_+}{t}\left(1 + 3 |\Qvec|^4
- 4\sqrt{6}\tr\Qvec^3\right) \right] \d V.
\end{equation}
The re-scaled bulk potential corresponds
to $f_B(\Qvec) - \min_{\Qvec\in S_0} f_B(\Qvec)$, where we have introduced an additive constant to make the bulk energy density non-negative.

We impose Dirichlet radial boundary conditions on the inner and outer radii as shown below:
\begin{equation}
\label{eq:8} \Qvec = \Qvec_b \quad \textrm{on } r=1 \textrm{ and } r=R
\end{equation}
where
\begin{equation*}
\label{eq:9} \Qvec_b := \sqrt{\frac{3}{2}}\left( \xhat \otimes
\xhat - \frac{\mathbf{I}}{3}\right).
\end{equation*}
The unit-vector, $\xhat := \frac{\mathbf{x}}{r}$ with
$r:=|\xvec|$, is the radial unit-vector. By definition, $\Qvec_b$
is perfectly uniaxial and is a minimum of the bulk potential,
i.e., it takes its values in the set defined by~\eqref{eq:3}.

We study the variational problem
\begin{equation} \label{pb:Landau-dG} \tag{LG}
 \min_{\Qvec\in\mathcal{A}} I[\Qvec] ,
\end{equation}
where $I$ is given by~\eqref{eq:6} and $\mathcal{A}$ is the
admissible class defined by
\begin{equation}
\label{eq:admissible}
\mathcal{A} := \left\{ \Qvec \in W^{1,2}\left(\Omega; \, S_0 \right)\colon \Qvec=\Qvec_b \quad \textrm{on } r=1 \textrm{ and } r=R \right\} .
\end{equation}
The corresponding Euler-Lagrange equations are given by
\begin{equation}
\label{eq:7} \Delta \Qvec_{ij} =
\frac{t}{2}\Qvec_{ij}\left(|\Qvec|^2 - 1 \right) +
\frac{h_+}{8}\left( 12 |\Qvec|^2 \Qvec_{ij} -
12\sqrt{6}\Qvec_{ip}\Qvec_{pj} + 4\sqrt{6}|\Qvec|^2 \delta_{ij}
\right).
\end{equation}
We are interested in locally stable equilibria, that is,
solutions of~\eqref{eq:7} for which the second variation of $I$ is
positive (see Subsection~\ref{subsec:local stability} and
Section~\ref{sec:t}), and global minimizers for the
problem~\eqref{pb:Landau-dG}.

\section{The Radial-Hedgehog Solution}
\label{sec:RH}

We define the radial-hedgehog solution to be a
minimizer of the Landau-de Gennes energy~\eqref{eq:6} in the class
of all radially-symmetric uniaxial $\Qvec$-tensors. This is
analogous to the definition of the radial-hedgehog solution on a
3D spherical droplet with radial boundary conditions, as
previously used in the literature~\cite{sonnet, sima2012, am2}.

We define the radial-hedgehog solution to be
\begin{equation}
\label{eq:10} \Hvec := \sqrt{\frac{3}{2}} h(r)\left( \xhat
\otimes \xhat - \frac{\mathbf{I}}{3}\right)
\end{equation}
where $h(r)$ is a minimizer of
\begin{equation}
\label{eq:11}
E[h] := \int_{1}^{R} \frac{r^2}{2}{\der{h}}^2 + 3 h^2 
+ t \, r^2 \left[ \frac{(1- h^2)^2}{8}+ \frac{h_+}{8t}\left(1 + 3 h^4 -4 h^3 \right) \right]~ \d r
\end{equation}
subject to the boundary conditions
\begin{equation}\label{eq:11bc} h(1) = h(R) = 1.\end{equation}
This is consistent with the Dirichlet conditions defined in~\eqref{eq:8}.
The admissible space for the variational problem in~\eqref{eq:11} is taken to be
\begin{equation*}
\label{eq:11b}
\mathcal{A}_h := \left\{ h \in L^2 \left([1, \, R]; \, \d r \right)\colon 
\der{h} \in L^2 \left([1, \, R]; \, r^2 \d r \right) \textrm{ and } \ h(1) = h(R)=1 \right\}.
\end{equation*}

The minimizing function, $h(r)\in \mathcal{A}_h$, is a solution of
the following second-order ordinary differential equation
\begin{equation}
\label{eq:12} 
\sder{h} + \frac{2}{r}\der{h} - \frac{6}{r^2} h = h f(h)
\end{equation}
where
\begin{equation} \label{f-alone}
 f(h) := \frac{t}{2} (h^2 - 1) + \frac{3h_+}{2}\left(h^2 - h\right) ,
\end{equation}
subject to~\eqref{eq:11bc}. 
One can check that $\Hvec$ thus defined is a solution
of the Euler-Lagrange equations in~\eqref{eq:7}, i.e. $\Hvec$
is a critical point of the Landau-de Gennes energy. In the subsequent sections, we investigate the local and global stability of $\Hvec$ as a function of the shell width, $(R-1)$, and the reduced temperature $t$.

\vspace{.25 cm}

\begin{proposition}\label{prop:1}
Define the function $\eta\colon [1, \, R] \to \mathbb{R}$ to be
\begin{equation} \label{eq:13} 
\eta(r) := \frac{1}{R^5 - 1} \left[(R^3 - 1)r^2 +
(R^2 - 1)\left(\frac{R}{r}\right)^3 \right].
\end{equation}
Then $\eta$ satisfies the following ordinary differential
equation:
\begin{equation}
\label{eq:14} \sder{\eta}  + \frac{2}{r} \der{\eta} - \frac{6}{r^2}\eta = 0
\end{equation}
subject to the boundary conditions $\eta(1)=\eta(R)=1$. There
exists a $R^* > 1$ such that
\begin{equation*}
\eta(r) \geq \frac{2}{3} \quad \textrm{for} \quad 1\leq r \leq R \leq R^*.
\label{eq:15}
\end{equation*}
\end{proposition}

\begin{proof} One can check by substitution that $\eta$, as
defined in~\eqref{eq:13}, is indeed a solution of~\eqref{eq:14}
subject to $\eta(1)=\eta(R)=1$. One can compute the minimum of
$\eta$ as a function of $R$:
an elementary computation shows that
\[
 \min_{1 \leq r \leq R} \eta (r) = \frac{5\left(R^3(R^2 - 1)\right)^{2/5}\left(R^3 - 1\right)^{3/5}}{2^{2/5}3^{3/5}\left(R^5 - 1\right)} \xrightarrow[R \to 1 \ ]{} 1 ,
\]
so there exists $R^* > 1$ such that
\begin{equation*} \label{eta min}
 \eta(r) \geq \frac23 \qquad \textrm{for }1 \leq r \leq R,
\end{equation*}
when $1 < R < R^*$.
\end{proof}


\begin{proposition}
\label{prop:2} The function $\eta$, defined in~\eqref{eq:13}, is a lower bound for the function
$h\colon [1, \, R] \to \Rr$ defined in
\eqref{eq:10}--\eqref{eq:12}, i.e.,
\begin{equation*}
\label{eq:16} \frac{2}{3}\leq \eta(r) \leq h(r) \leq 1 \quad \textrm{for} \quad 1\leq
r\leq R \leq R^*.
\end{equation*}
\end{proposition}
\begin{proof} The proof is parallel to the proof in the
two-dimensional case, presented in~\cite{golovaty}. We define the
function
\begin{equation*}
\label{eq:nu1}
\nu(r) := \eta(r) - h(r) \quad \textrm{for} \quad 1\leq r \leq R
\end{equation*}
where $\nu(1) = \nu(R) = 0$. We proceed by contradiction.
We assume that $\nu$ has a positive maximum for $r^*\in \left(1, R \right)$.
The function $\nu$ is a solution of the following second-order differential equation
\begin{equation}
\label{eq:nu2}
\sder{\nu} + \frac{2}{r} \der{\nu} - \frac{6}{r^2} \nu = - h f(h) .
\end{equation}
The function $h(r)$ satisfies the bounds, $0\leq h(r)\leq 1$; these bounds are established in~\cite{ejam, sima2012}.
Therefore, the right-hand side of~\eqref{eq:nu2} is non-negative for all $1\leq r \leq R$.
This is enough to exclude a positive interior maximum 
and hence, we deduce that
\begin{equation*}
\label{eq:nu4}
\nu(r) = \eta(r) - h(r)  \leq 0 \quad \textrm{for} \quad 1\leq r \leq R,
\end{equation*} as required. \end{proof}


\subsection{Energy Expansion}

Let $\Qvec\in \mathcal{A}$ be an arbitrary $\Qvec$-tensor in our admissible space.
Then $\Qvec$ can be written as
\begin{equation*}
\label{eq:17b} 
\Qvec = \Hvec + \Vvec
\end{equation*}
with $\Vvec \in W^{1,2}\left(\Omega; \, S_0 \right)$ and
\begin{equation*}
\label{eq:18b} \Vvec = 0 \quad \textrm{on }r=1\textrm{ and } r=R,
\end{equation*}
since $\Qvec - \Hvec = 0$ on the boundaries. 
The first step is to compute an energy expansion
for $\Qvec$ in terms of $\Hvec$ and $\Vvec$; direct computations
show that
\begin{align*}
\label{eq:19b}
& |\Qvec|^2 = h^2 + 2 \left(\Hvec\cdot\Vvec\right) +  |\Vvec|^2 \\
&|\Qvec|^4 = h^4 + 4 h^2 \left(\Hvec\cdot\Vvec \right)+ 2 h^2 |\Vvec|^2 + 4\left(\Hvec\cdot\Vvec \right)^2 +
4(\Hvec\cdot\Vvec)|\Vvec|^2 +  \left| \Vvec \right|^4 \\
&\left(1 - |\Qvec|^2 \right)^2 = \left(1 - h^2\right)^2 + 4\left(\Hvec\cdot\Vvec \right)\left(h^2 - 1\right) + \\
&\qquad \qquad \quad \ + 2 |\Vvec|^2\left(h^2 - 1 \right) + 4\left(\Hvec\cdot\Vvec \right)^2 + 4\left(\Hvec\cdot\Vvec \right)\left| \Vvec \right|^2 +
\left| \Vvec \right|^4 \\
& \tr\Qvec^3 = \frac{h^3}{\sqrt{6}} + 3\tr\left(\Hvec^2\Vvec\right) + 3\tr\left(\Hvec\Vvec^2\right) + \tr\Vvec^3 \\
&\left| \grad \Qvec \right|^2 = \left|\grad \Hvec\right|^2 +  2\left(\grad \Hvec \cdot \grad \Vvec \right) + \left|\grad \Vvec \right|^2.
\end{align*}
We note that
\begin{equation*} \label{eq:20b} 
\tr\left( \Hvec\Vvec^2\right) = \sqrt{\frac{3}{2}}h\left( r \right) 
\left[ \left(\xhat\cdot \Vvec\right)^2 - \frac{|\Vvec|^2}{3} \right].
\end{equation*}

The Landau-de Gennes energy of $\Qvec$ can then be written as
\begin{equation} \label{eq:21}
 \begin{split}
	I[\Qvec] &= I[\Hvec] + \int_{\Omega}\grad\Hvec \cdot \grad \Vvec + 
		\frac{t}{2}\left(\Hvec \cdot \Vvec\right) \left(h^2 - 1 \right)~\d V \\
	& + \frac{h_+}{8}\int_\Omega 12 h^2 \left(\Hvec\cdot \Vvec \right) - 12\sqrt{6} \tr\left(\Hvec^2\Vvec\right)~\d V \\
	& + \int_{\Omega}\frac{1}{2}|\grad \Vvec|^2 + 
		\frac{t}{8}\left(4\left(\Hvec\cdot \Vvec \right)^2 + 2|\Vvec|^2\left( h^2 - 1 \right) \right)~\d V \\
	& + \frac{h_+}{8} \int_{\Omega} 6 h^2 |\Vvec|^2 + 12 \left(\Hvec\cdot \Vvec \right)^2 - 
		12\sqrt{6}\left(\Hvec\Vvec^2\right)~\d V \\
	& +  \int_{\Omega}\frac{t}{2}\left(\Hvec\cdot \Vvec\right)|\Vvec|^2 + 
		\frac{h_+}{8}\left( 12\left(\Hvec\cdot \Vvec\right)|\Vvec|^2 - 4\sqrt{6}\tr\Vvec^3 \right)~\d V \\
	& +  \int_{\Omega} \frac{t}{8}|\Vvec|^4 + \frac{3 h_+}{8}|\Vvec|^4~\d V.
 \end{split}
\end{equation}
The sum of the first and the second integral (that is, all the linear terms in $\Vvec$)
vanishes since $\Hvec$ is a critical point of the Landau-de Gennes energy.

We use the following basis for the space $S_0$, as introduced in
\cite{ignat2014}. Let $\nvec= \xhat$ and let $\left(\nvec, \mvec,
\pvec \right)$ denote an orthonormal basis for $\Rr^3$. In terms
of spherical polar coordinates, $\left(r, \theta, \phi \right)$,
we have
\begin{align*} \label{eq:vectors}
\nvec &:= \left( \sin\theta \cos\phi, \sin\theta \sin\phi, \cos\theta \right) \\
\mvec &:= \left( \cos\theta\cos\phi, \cos\theta\sin\phi, -\sin\theta \right) \\
\pvec &:= \left(- \sin\phi, \cos\phi, 0 \right).
\end{align*}
for~$0\leq \theta \leq \pi$ and~$0\leq \phi < 2\pi$.
Following the paradigm in~\cite{ignat2014}, we define
\begin{gather*} 
  \Evec := \nvec\otimes\nvec - \frac{\mathbf{I}}{3} ; \qquad
  \Fvec := \nvec\otimes\mvec + \mvec\otimes\nvec ; \\
  \Gvec := \nvec\otimes\pvec + \pvec\otimes\nvec ; \qquad
  \Xvec := \mvec\otimes\pvec + \pvec\otimes\mvec ; \qquad 
  \Yvec := \mvec\otimes\mvec - \pvec\otimes\pvec
\end{gather*}
where $|\Evec|^2 = 2/3$ and $|\Fvec|^2 = |\Gvec|^2 = |\Xvec|^2 = |\Yvec|^2 = 2$. Then any arbitrary $\Vvec \in S_0$ can be written as
\begin{equation} \label{eq:24}
\Vvec = v_0 \Evec + v_1 \Fvec + v_2 \Gvec + v_3
\Xvec + v_4 \Yvec
\end{equation}
for functions $v_0, v_1, v_2, v_3, v_4\colon \Omega \to \Rr$ and all five functions vanish on $r=1$ and $r=R$.

The key quantities in~\eqref{eq:21} can be written in terms of
$v_0,v_1,\ldots, v_4$, so that  the energy difference, $I[\Qvec] - I[\Hvec]$, is
%
\begin{equation} 
\label{eq:main2}
\begin{split}
	I[\Qvec] &- I[\Hvec] = \int_{\Omega} \frac{1}{2}|\grad \Vvec|^2 
		+ \frac{t}{4} |\Vvec|^2 \left(h^2 - 1 \right) + \frac{t}{3} h^2 v_0^2 ~\d V  \\
	& + \int_{\Omega} \frac{h_+ v_0^2}{2} \left( 3 h^2 - 2 h \right)
		+ \frac{3 h_+}{2}\left(h^2 + 2h \right) \left(v_3^2 + v_4^2 \right) 
		+ \frac{3 h_+}{2}\left( h^2 - h \right)\left( v_1^2 + v_2^2 \right)~\d V \\
	& + \left( \frac{t}{\sqrt{6}} + \sqrt{\frac{3}{2}}h_+ \right)
		\int_{\Omega} h v_0 \left( \frac{2}{3}v_0^2 + 2 \left(v_1^2 + v_2^2 + v_3^2 + v_4^2 \right) \right)~\d V \\
	& - \frac{\sqrt{6}h_+}{2} \int_{\Omega}\frac{2}{9}v_0^3 + v_0\left(v_1^2 + v_2^2 \right) 
		+ 6 v_1 v_2 v_3 + 3 v_4 \left(v_1^2 - v_2^2 \right)- 2v_0\left(v_3^2 + v_4^2 \right)~\d V  \\ 
	& +\frac{t + 3h_+}{8} \int_{\Omega} \frac{4}{9}v_0^4
		+ 4\left(v_1^2 + v_2^2 + v_3^2 + v_4^2 \right)^2 
		+ \frac{8}{3}v_0^2\left(v_1^2 + v_2^2 + v_3^2 + v_4^2 \right)~\d V. \\
\end{split}
\end{equation}


\subsection{Local Stability}
\label{subsec:local stability}

We compute the second variation of the Landau-de Gennes energy~\eqref{eq:6} about the radial-hedgehog solution, $\Hvec$ (defined in~\eqref{eq:10}--\eqref{eq:12}).
We recall that the second variation is, by definition,
\begin{equation*}
 \delta^2 I[\Hvec] := \dfrac{\d^2}{\d s^2}_{|s = 0} \, I[\Hvec + s\Vvec]
\end{equation*}
where $\Vvec\in W^{1, 2}_0(\Omega; \, S_0)$ is a fixed perturbation
(see~\cite{ejam, sima2012} for similar computations on a 3D
droplet). By inspecting Equation~\eqref{eq:main2} and collecting all
the quadratic terms in $v_0, v_1, v_2, v_3, v_4$, it is straightforward to verify
that the second variation is given by
\begin{equation}
\label{eq:27}
\begin{split}
\delta^2 I[\Hvec] &= \int_{\Omega} |\grad \Vvec|^2
+ \frac{t}{2}|\Vvec|^2\left(h^2 - 1 \right)~\d V \\
&+ \int_{\Omega} \frac{2t}{3}h^2 v_0^2 + h_+ v_0^2
\left( 3 h^2 - 2 h \right) + 3 h_+\left(h^2 + 2h\right)
\left(v_3^2 + v_4^2 \right)~\d V \\
&+ \int_{\Omega}3 h_+\left( h^2 - h\right)\left(v_1^2 + v_2^2 \right)~\d V.
\end{split}
\end{equation}

\vspace{.25 cm}

\begin{theor}
\label{prop:3} The radial-hedgehog solution, $\Hvec$, is a locally
stable equilibrium of the Landau-de Gennes energy~\eqref{eq:6}, in
the space $\mathcal{A}$ i.e.
\begin{equation*} \label{eq:27b}
\delta^2 I[\Hvec] > 0
\end{equation*}
for all $ t\geq 0$ and
\begin{equation*} \label{eq:27c}
1< R < \min \left\{ R^*, \, 1 + \frac{\pi}{\sqrt{6}} \right\},
\end{equation*}
where $R^*$ has been defined in Proposition~\ref{prop:1}.
\end{theor}

\vspace{.5 cm}

\begin{proof} The proof follows from a Hardy-type
trick. We start with the integral expression~\eqref{eq:27}. We
recall from Proposition~\ref{prop:2} that for $R < R^*$, we have
\begin{equation*} \label{eq:27d}
\frac{2}{3} \leq h(r) \leq 1 \quad \textrm{for any } 1\leq r \leq R
\end{equation*}
so that $3 h^2 - 2 h \geq 0$ for $r\in\left[1, \, R \right]$.
Therefore, there are two problematic non-positive terms above in
\eqref{eq:27}: $\frac{t}{2}|\Vvec|^2\left(h^2 - 1 \right)$ and $3 h_+\left( h^2 - h\right)\left(v_1^2 + v_2^2 \right)$ . We combine the two
non-positive terms so that the second variation is bounded from below by
\begin{equation}
\label{eq:b1}
\delta^2 I \geq \int_{\Omega} |\grad \Vvec|^2 + f(h) |\Vvec|^2~\d V
\end{equation}
(the function~$f(h)$ has been defined in Equation~\eqref{f-alone}).
An arbitrary $\Vvec$ can be written as
\begin{equation*}
\label{eq:b2}
\Vvec(\xvec) = h(r) \Vbar (\xvec)
\end{equation*}
where $\Vbar$ vanishes on $r=1$ and $r=R$, since
$h$ is strictly positive for $1\leq r \leq R$. We use integration by parts to obtain
\begin{equation}
\label{eq:b4}
\int_{\Omega} |\grad \Vvec|^2~\d V =
\int_{0}^{2\pi}\int_{0}^{\pi}\int_{1}^{R}\left(h^2 r^2 |\grad
\Vbar|^2 - |\Vbar|^2 h \sder{h} r^2  - 2 h \der{h} r |\Vbar|^2 \right) \sin\theta \, \d r \, \d\theta \, \d\phi.
\end{equation}
Recalling the ordinary differential equation for $h(r)$ in~\eqref{eq:12}, we see that
\begin{equation}
\begin{split} \label{eq:b5} 
\int_{\Omega} f(h) |\Vvec|^2~\d V =
\int_{0}^{2\pi} \int_{0}^{\pi}\int_{1}^{R}&\left(\sder{h} + 
\frac{2}{r}\der{h} - \frac{6}{r^2} h \right) h |\Vbar|^2 r^2 \sin\theta \, \d r \, \d\theta \, \d\phi.
\end{split}
\end{equation}
Combining~\eqref{eq:b1},~\eqref{eq:b4} and~\eqref{eq:b5}, we
obtain
\begin{equation}
\label{eq:b6} \delta^2 I \geq
\int_{0}^{2\pi}\int_{0}^{\pi}\int_{1}^{R} h^2(r)\left( r^2 |\grad
\Vbar|^2 - 6 |\Vbar|^2 \right)\sin\theta \, \d r \, \d\theta \, \d\phi.
\end{equation}
We now use $r\geq 1$ and Wirtinger's inequality~\cite{evans}
\[
\int_{1}^{R} {\der{v}}^2 \, \d r \geq \frac{\pi^2}{(R-1)^2} \int_{1}^{R} v^2 \,\d r
\]
for any function
$v\colon [1, \, R] \to \Rr$ such that $v(1) = v(R)=0$, to see that $\delta^2 I > 0$ for
\begin{equation*}
\label{eq:b7}
\left( R - 1 \right)^2 < \frac{\pi^2}{6}. \tag*{\qedhere}
\end{equation*}
\end{proof}

\section{On the Minimality of the Hedgehog when \texorpdfstring{$R - 1$}{R - 1} is Small}%
\label{sec:global}

%

This section is devoted to the proof of Theorem~\ref{th:1}, i.e., we assume that $R$ satisfies 
\begin{equation} \label{hp: R small}
 R < \min\left\{R_0 := \exp\left( \frac{4\pi^2}{23}\right), R^* \right\}
\end{equation}
where~$R^*$ is defined in Proposition~\ref{prop:1},
and prove that the radial-hedgehog is energy minimizing.
As a preliminary remark, we point out that the smallness assumption~\eqref{hp: R small} on $R - 1$ and Proposition~\ref{prop:1} imply
\begin{equation} \label{hp: h large}
h(r) \geq \frac23 \quad \textrm{for} \quad 1 \leq r \leq R .
\end{equation}

Take an admissible field $\Qvec \in \mathcal{A}$ and set $\Vvec :=
\Qvec - \Hvec \in W^{1, 2}_0(\Omega; \, S_0)$. The functions $v_0$, $v_1$\ldots, $v_4$, are the coordinates of $\Vvec$ with
respect to the basis $\Evec$, $\Fvec$, $\Gvec$, $\Xvec$, $\Yvec$:
\[
 \Vvec = v_0 \Evec + v_1 \Fvec + v_2 \Gvec + v_3 \Xvec + v_4 \Yvec .
\]
We have an expression for the energy difference $I[\Qvec]- I[\Hvec]$, namely, Equation~\eqref{eq:21}:
\[
\begin{split}
 I[\Qvec] - I[\Hvec] &= \int_\Omega \bigg\{ \frac12 \abs{\grad \Vvec}^2 + \frac t4 \abs{\Vvec}^2 (h^2 - 1) + \frac t2 (\Hvec\cdot\Vvec)^2 \\
               &\qquad +\frac{h_+}{8} \left(6h^2 \abs{\Vvec}^2 + 12 (\Hvec\cdot\Vvec)^2 - 12\sqrt{6} \left(\Hvec\Vvec^2\right)\right) \\
               &\qquad + \frac t2 (\Hvec\cdot\Vvec) \abs{\Vvec}^2 + \frac{h_+}{8} \left(12(\Hvec\cdot\Vvec)\abs{\Vvec}^2 - 4\sqrt{6} \tr\Vvec^3\right) \\
               &\qquad + \frac t8 \abs{\Vvec}^4 + \frac{3h_+}{8} \abs{\Vvec}^4 \bigg\}~\d V.
\end{split}
\]
A direct computation shows that
\[
\begin{split}
 - 12\sqrt{6} \tr\left(\Hvec\Vvec^2\right) &=
-4h \left( v_0^2 - 9 v_3^2 - 9 v_4^2 \right) - 6h \abs{\Vvec}^2 ,
\end{split}
\]
so
\begin{equation} \label{second variation V}
\begin{split}
 I[\Qvec] - I[\Hvec] &= \int_\Omega \bigg\{ \frac12 \abs{\grad \Vvec}^2 + \frac t4 \abs{\Vvec}^2 (h^2 - 1) + \frac t2 (\Hvec\cdot\Vvec)^2 \\
               &\qquad +\frac{3h_+}{4}\abs{\Vvec}^2 (h^2 - h)  + \frac{h_+}{8} \left(12 (\Hvec\cdot\Vvec)^2 -4h \left( v_0^2 - 9 v_3^2 - 9 v_4^2 \right)\right) \\
               &\qquad + \frac t2 (\Hvec\cdot\Vvec) \abs{\Vvec}^2 + \frac{h_+}{8} \left(12(\Hvec\cdot\Vvec)\abs{\Vvec}^2 - 4\sqrt{6} \tr\Vvec^3\right) \\
               &\qquad + \frac t8 \abs{\Vvec}^4 + \frac{3h_+}{8} \abs{\Vvec}^4 \bigg\}~\d V \\
               &= \int_\Omega \bigg\{ \frac12 \abs{\grad \Vvec}^2 + f(h) \abs{\Vvec}^2 
               + h_+ \left(-\frac{\sqrt{6}}{2} \tr\Vvec^3 + \frac h2 \left( - v_0^2 + 9 v_3^2 + 9 v_4^2 \right)\right) \\
               &\qquad + \frac{t + 3h_+}{8} \left(2(\Hvec\cdot\Vvec) + \abs{\Vvec}^2\right)^2  \bigg\}~\d V .
\end{split}
\end{equation}
To deal with the first two terms, we write $\Vvec= h\Wvec$, $v_i = h w_i$ and use the Hardy
decomposition trick again. With computations similar to
\eqref{eq:b5}--\eqref{eq:b6}, we obtain
\begin{equation} \label{energy difference}
\begin{split}
 I[\Qvec] - I[\Hvec]
               &= \int_{\Omega} \bigg\{ h^2 \left(\frac{1}{2} \abs{\grad \Wvec}^2 - \frac {3}{r^2} \abs{\Wvec}^2\right) + h_+  h^3 \psi(\Wvec)\\
               &\qquad + \frac{t + 3h_+}{8}h^4\left(2\left(\frac{\Hvec}{h}\cdot\Wvec\right) + \abs{\Wvec}^2\right)^2 \bigg\}~\d V
\end{split}
\end{equation}
where
\begin{equation} \label{psi}
\begin{split}
 \psi(\Wvec) &:= -\frac{\sqrt{6}}{2} \tr\Wvec^3 - \frac 12 w_0^2 + \frac 92 w_3^2 + \frac 92 w_4^2 \\
          &= -\frac12 w_0^2 + \frac92 \left(w_3^2 + w_4^2 \right) + \sqrt 6 w_0 \left(w_3^2 + w_4^2\right) \\
          &\qquad + \frac{3\sqrt 6}{2} w_4\left(w_2^2 - w_1^2\right) - 3\sqrt 6 w_1w_2w_3 - \frac{\sqrt 6}{2} w_0 \left(w_1^2 + w_2^2\right) - \frac{\sqrt 6}{9} w_0^3 .
\end{split}
\end{equation}

In order to prove Theorem~\ref{th:1}, we need to show
$I[\Qvec] - I[\Hvec]\geq 0$ for any admissible $\Qvec$, with
equality if and only if $\Qvec=\Hvec$. In the following lemmas, we
prove that $\psi(\Wvec)$ is non-negative. 
We then use a Poincar\'e-type inequality to demonstrate the positivity of the bracketed integral, 
$\int_{\Omega} h^2 \left(\frac{1}{2} \abs{\grad \Wvec}^2 - \frac {3}{r^2} \abs{\Wvec}^2\right)~\d V$, above for small $R - 1$ .
This completes the proof of the theorem.

\begin{lemma} \label{lemma: lower bound psi}
 Let $\psi$ be defined by Formula~\eqref{psi}, we have
 \[
  \psi(w_0, \, w_1, \, w_2, \, \, w_3, \, w_4) \geq \psi\left(w_0, \, \sqrt{w_1^2 + w_2^2}, \, 0, \, 0, \, \sqrt{w_3^2 + w_4^3} \right)
 \]
 for all $(w_0, \, w_1, \, w_2, \, w_3, \, w_4)\in\Rr^5$.
\end{lemma}
\begin{proof} Thanks to~\eqref{psi}, the lemma reduces to proving that 
\begin{equation} \label{ineq psi}
 \frac{3\sqrt 6}{2} w_4\left(w_2^2 - w_1^2\right) - 3\sqrt 6 w_1w_2w_3 \geq - \frac{3\sqrt 6}{2} \sqrt{w_3^2 + w_4^2}\left(w_1^2 + w_2^2\right) .
\end{equation}
Let us consider the change of variables given by
\[
 w_1 = \rho \cos\theta \cos\varphi_1 \qquad w_2 = \rho \cos\theta \sin\varphi_1 \qquad
 w_3 = \rho \sin\theta \cos\varphi_2 \qquad w_4 = \rho \sin\theta \sin\varphi_2 ,
\]
where
\[
 \rho > 0, \, \qquad 0 \leq \theta \leq \frac{\pi}{2}, \qquad 0 \leq \varphi_1, \, \varphi_2 < 2\pi .
\]
Firstly, we remark that this formula defines an admissible change
of variable, in the sense that $(\rho, \, \theta, \, \varphi_1, \,\varphi_2) \mapsto (w_1, \, w_2, \, w_3, \, w_4)$ gives a
one-to-one and onto mapping $(0, \, +\infty) \times [0, \, \pi/2] \times [0, \, 2\pi)^2 \to \Rr^4 \setminus \{0\}$. 
Secondly, we write the left hand side of~\eqref{ineq psi} in terms of the new variables and obtain
\[
\begin{split}
 \frac{3\sqrt 6}{2} w_4\left(w_2^2 - w_1^2\right) &- 3\sqrt 6 w_1w_2w_3 \\
 &= -\frac{3 \sqrt 6}{2} \rho^3 \sin\theta \cos^2\theta \left(\cos(2\varphi_1)\sin\varphi_2 + \sin(2\varphi_1)\cos\varphi_2 \right) \\
 &= -\frac{3 \sqrt 6}{2} \rho^3 \sin\theta \cos^2\theta \sin\left(2\varphi_1 + \varphi_2\right) \\
 &\geq -\frac{3 \sqrt 6}{2} \rho^3 \sin\theta \cos^2\theta ,
\end{split}
\]
which is precisely the right hand side of~\eqref{ineq psi}.
\end{proof}

\begin{lemma} \label{lemma: psi positive}
 If~\eqref{hp: h large} holds, then
 \[
  \psi(\Wvec) + \frac{3h}{8} \left(2\frac{\Hvec}{h}\cdot\Wvec + \abs{\Wvec}^2\right)^2 \geq 0 .
 \]
\end{lemma}
\begin{proof} It is convenient to express the function $\psi$ in terms of
a new set of variables for the proof of this lemma. From Lemma~\ref{lemma: lower bound psi}, we can assume WLOG that $w_2 = w_3 = 0$.

Let
\begin{equation} \label{X}
X := \sqrt{\frac23} (w_0 + 3w_4)
\end{equation}
and
\begin{equation} \label{epsilon}
\epsilon := 2\frac{\Hvec}{h}\cdot\Wvec + \abs{\Wvec}^2 = \frac23 w_0^2 +
2\sqrt{\frac23}w_0 + 2w_1^2 + 2w_4^2 .
\end{equation}
Substituting~\eqref{X} and~\eqref{epsilon} into the right hand side of~\eqref{psi}, we obtain
\begin{equation} \label{psi X}
 \psi(\Wvec) = \frac14 \left( X^3 + 3 X^2 - 3\epsilon X\right) .
\end{equation}
Thus, $\psi$ reduces to a polynomial of degree three in the variables~$X$ and~$\epsilon$.

Our goal is to minimize~$\psi$ and we need to demarcate the
relevant ranges for the variables~$X$ and~$\epsilon$. Note that
\[
 \epsilon = \abs{\frac{\Hvec}{h} + \Wvec}^2 - 1 = \abs{\frac{\Qvec}{h}}^2 -1 \geq -1
\]
by definition. We can deduce the following inequality from Equation~\eqref{epsilon}:
\begin{equation}\label{ellipse}
 \frac23 \left(w_0 + \sqrt{\frac 32}\right)^2 + 2w_4^2 \leq 1 + \epsilon ,
\end{equation}
from which it is clear that~\eqref{ellipse} represents a region
bounded by an ellipse in the $(w_0, \, w_4)$-plane.
We denote that
region by $\Sigma$. Then, $X$ can take any value between the
minimum and the maximum of the function $F\colon (w_0, \, w_4)
\mapsto \sqrt{2/3}(w_0 + 3w_4)$ over $\Sigma$. By the Lagrange
multiplier theorem, at the extrema the tangent lines to the
ellipse $\partial \Sigma$ have equation $\sqrt{2/3}(w_0 + 3w_4) =
c$. Thus, the minimum and the maximum value of $F$ over $\Sigma$
are exactly the values of $c$ for which the line $\sqrt{2/3}(w_0 +
3w_4) = c$ is tangent to $\partial \Sigma$. These values can be
computed, e.g., by forcing the system for $(w_0, \, w_4)$
\[
 \begin{cases}
  \dfrac23 w_0^2 + 2\sqrt{\dfrac23}w_0 + 2w_4^2 = \epsilon \\
  \sqrt{\dfrac 23}(w_0 + 3w_4) = c
 \end{cases}
\]
to have a unique solution. We manipulate the two relations above to 
conclude that
\begin{equation} \label{X range}
 -1 - 2\sqrt{\epsilon + 1} \leq X \leq -1 + 2 \sqrt{\epsilon + 1} .
\end{equation}

Next, we minimize the right hand side of~\eqref{psi X}, as a
function of $X$, in the range~\eqref{X range}. We obtain
\[
 \psi(\Wvec) \geq \psi\left(- 1+ \sqrt{\epsilon + 1}\right) = \frac34 \epsilon + \frac 12 - \frac 12 (\epsilon + 1)^{3/2} ,
\]
hence, if the condition~\eqref{hp: h large} is satisfied 
\[
 \psi(\Wvec) + \frac{3h}{8} \left( 2\frac{\Hvec}{h}\cdot\Wvec + \abs{\Wvec}^2 \right)^2 \geq \frac14 \epsilon^2 + \frac34 \epsilon + \frac 12 - \frac 12 (\epsilon + 1)^{3/2} =: G(\epsilon).
\]
Finally, we need to show that the function $G$ is non negative on
$[-1, \, +\infty)$. A standard analysis shows that $G$ has a global
minimum on $[-1, \, +\infty)$, which is either $\epsilon = -1$ or
an interior critical point. Now, $G(-1) = 0$, and there are two
critical points for $G$: $\epsilon = -3/4$ (which is a local
maximum) and $\epsilon = 0$. Therefore, $G(\epsilon) \geq 0$ for
every $\epsilon \geq -1$.
\end{proof}

\begin{lemma} \label{lemma: hardy ineq}
For all $R > 1$ there exists a (optimal) constant $C_H(R) > 0$
such that, for all $v\in H^1_0(1, \, R)$, we have
\[
  \int_1^R {\der{v}}^2 r^2 \, \d r \geq C_H(R) \int_1^R v^2 \, \d r .
\]
Moreover, $C_H(R) > 1/4$ for all $R > 1$.
\end{lemma}
\begin{proof}
We consider the following minimization problem with constraints:
\[
 C_H(R) := \min\left\{ \int_1^R r^2 {\der{v}}^2 \, \d r \colon v\in H^1_0(1, \, R), \, \int_1^R v^2 \, \d r = 1 \right\} .
\]
Using standard methods in the calculus of variations, one can check that a minimizer exists.
By Lagrange's multiplier theorem, any minimizer solves the eigenvalue problem
\begin{equation} \label{eigen}
\begin{cases}
 - \dfrac{\d}{\d r} \left( r^2 \der{v}(r) \right) = \lambda v (r) \\
 v(1) = v(R) = 0,
 \end{cases}
\end{equation}
and, in particular,
\[
 r^2 \sder{v} + 2 r \der{v} + \lambda r = 0 .
\]
This equation can be easily solved, e.g., with the change of variable $r = e^t$, $u(t) = v(r)$.
One finds a necessary and sufficient condition for the existence
of a non-trivial solution $v\not\equiv 0$ to~\eqref{eigen}, namely
that
\[
 \lambda = \lambda_k (R) := \frac{k^2\pi^2}{\log R} + \frac14 \qquad \textrm{for some } k\in\N\setminus \{0\} .
\]
Thus, the $\lambda_k$'s are the eigenvalues for~\eqref{eigen} and
\begin{equation} \label{optimal hardy constant}
C_H(R) = \lambda_1(R) = \frac{\pi^2}{\log R} + \frac14 .
\end{equation}
This proves the lemma.
\end{proof}

The proof of Theorem~\ref{th:1} now follows from
the previous lemmas.

\begin{proof}[Proof of Theorem~\ref{th:1}]
 To prove the minimality of the hedgehog, we must show that
 \[
  I[\Qvec] - I[\Hvec] \geq 0 .
 \]
 By Equation~\eqref{energy difference}, we have
 \[
  \begin{split}
   I[\Qvec] - I[\Hvec]
               &\geq \int_{\Omega} h^2 \left(\frac{1}{2} \abs{\grad \Wvec}^2 - \frac {3}{r^2} \abs{\Wvec}^2\right)~\d V \\
               &\qquad + \int_\Omega h_+ h^3 \left\{\psi(\Wvec) + \frac{3h}{8}\left(2\frac{\Hvec}{h}\cdot\Wvec + \abs{\Wvec}^2\right)^2 \right\}~\d V \, .
  \end{split}
 \]
By virtue of~\eqref{hp: h large} and Lemma~\ref{lemma: psi
positive}, the second integral is non negative and we obtain
\[
 I[\Qvec] - I[\Hvec] \geq \int_{\Omega} \frac{4}{9}\left(\frac 12 \abs{\grad \Wvec}^2 - \frac {3}{r^2} \abs{\Wvec}^2\right)~\d V.
\]
 We can now write the integral using spherical coordinates, apply Fubini's theorem
 and Lemma~\ref{lemma: hardy ineq} to get
\[
 I[\Qvec] - I[\Hvec] \geq \frac{4}{9}\left(\frac12 C_H(R) - 3\right) \int_{\Omega} \frac{1}{r^2} \abs{\Wvec}^2~\d V .
\]
The constant $C_H(R)$ is given explicitly by~\eqref{optimal hardy
constant}. Finally, from assumption~\eqref{hp: R small},
we have
\[
 \frac12 C_H(R) - 3 = \frac{\pi^2}{2\log R} - \frac{23}{8} > 0.
\]
for $R < \min\left\{R_0, \, R^* \right\}$.
Hence, we conclude that $I[\Qvec] - I[\Hvec] \geq 0$, with equality if and only if~$\Qvec = \Hvec$.
\end{proof}

\section{Minimality of the Hedgehog for Large \texorpdfstring{$t$}{t}}
\label{sec:t}

This section is devoted to the proof of Theorem~\ref{th:2}, i.e., 
showing that the radial-hedgehog is energy-minimizing for all $R > 1$ (independent of $t$) and sufficiently large $t$.
As a preliminary step, we adapt the proof by Ignat et al.
\cite{ignat2014} and prove that the radial-hedgehog is locally
stable (i.e., the second variation of the energy is positive) when
the temperature $t$ is large enough, without restriction on $R -
1$. We first show that the temperature $t$ uniformly controls the
order parameter, $h$, for large enough $t>0$.
\begin{lemma} \label{lemma: h}
 Let $h\in H^1(1, \, R)$ be a minimizer of
 \begin{equation} \label{h energy}
  E[h] := \int_1^R \left\{ \frac 12 {\der{h}}^2 + \frac{3}{r^2} h^2 + \frac t8 (1 - h^2)^2 
  + \frac{h_+}{8} (1 + 3h^4 - 4 h^3)\right\} \, r^2 \d r,
 \end{equation}
 with the boundary conditions $h(1) = h(R) = 1$.
 Then, $0 < h \leq 1$ for all $t \geq 0$. Moreover,
 \[
  h \to 1 \qquad \textrm{uniformly as } t\to \infty .
 \]
\end{lemma}
\begin{proof} The bounds $0< h \leq 1$ are easily established
\cite{ejam,am2}. Indeed, $h\geq 0$ from the energy minimality of
$h$ (refer to~\eqref{h energy}). The function $h \leq 1$, as an
immediate consequence of the maximum principle. We can easily
prove that $h > 0$. Indeed, we assume that there exists a point
$r_1$ such that $h(r_1) = 0$. Since we know that $h \geq 0$, $r_1$
must be a minimum point for $h$, so $\der{h}(r_1) = 0$. Then we
apply the classical well-posedness theory for Cauchy problems for
ODE's and conclude that $h \equiv 0$, which contradicts the
boundary conditions $h(1) = h(R) = 1$. Thus, we must have $h > 0$.

Finally, we check the uniform convergence of $h$ as $t \to
\infty$. Let $r_{\mathrm{min}}\in (1, \, R)$ be a minimum point
for $h$. We have $\sder{h}(r_{\mathrm{min}}) \geq 0$ and
$\der{h}(r_{\mathrm{min}}) = 0$. Therefore, by Equation
\eqref{eq:12},
\[
 - \frac{6}{r_{\mathrm{min}}^2} h(r_{\mathrm{min}}) \leq  f(h(r_{\mathrm{min}})) h(r_{\mathrm{min}}) \leq \frac {t}{2} \left(h^2(r_{\mathrm{min}}) - 1\right) h(r_{\mathrm{min}}) .
\]
We divide by $h(r_{\mathrm{min}}) > 0$ and obtain
\[
 1 - h^2(r_{\mathrm{min}}) \leq \frac{12}{t \, r_{\mathrm{min}}^2} \leq \frac{12}{t} .
\]
Thus,
\[
 1 \geq h \geq \sqrt{1 - \frac{12}{t}} \to 1 \qquad \textrm{as } t \to +\infty \textrm{, uniformly on } (1, \, R) . \tag*{\qedhere}
\]
\end{proof}

The second step in our analysis for large $t$ is the study of the
second variation of the energy. Recall that, given a variation
$\Vvec \in W^{1, 2}_0(\Omega; \, S_0)$ (i.e., $\Qvec = \Hvec +
\Vvec$), the second variation is given by
\begin{equation} \label{second variation}
 \frac12 \delta^2 I[\Hvec] = \int_\Omega \left\{ \frac{1}{2} \abs{\nabla \Vvec }^2 + \frac13 f_0(h) v_0^2 + f_2(h) (v_1^2 + v_2^2) + f_4(h) (v_3^2 + v_4^2) \right\}~\d V
\end{equation}
(see Equation~\eqref{eq:27}), where $v_0, \, \ldots, \, v_4$ are
the coordinates of $\Vvec$ with respect to the basis we have chosen and
\begin{equation} \label{f}
\begin{split}
 f_0 (h) &:= \frac t2 \left(3h^2 - 1\right) + \frac{3h_+}{2} \left(3h^2 - 2h\right) \\
 f_2 (h) &:= \frac t2 \left(h^2 - 1\right) + \frac{3h_+}{2} \left(h^2 -h\right) \\
 f_4 (h) &:= \frac t2 \left(h^2 - 1\right) + \frac{3 h_+}{2} \left(h^2 + 2h \right) .
\end{split}
\end{equation}
Note that~$f_2$ coincides with the function~$f$ defined by Equation~\eqref{f-alone}.

We want to show that the second variation is positive for every
choice of $\Vvec$ and large enough $t$. In~\cite{ignat2014}, it is
shown that the analysis of $\delta^2 I[\Hvec]$ can be reduced to
the study of the simpler functionals $\phi_{0, i}$, defined for
$i\in\N$ by
\begin{equation} \label{phi_{0,i}}
\begin{split}
\phi_{0,0} [v_0]
    &:= \frac{2}{3}\int_1^R \left\{|\der{v_0}|^2 + \frac{6}{r^2} v_0^2 + f_0(h)v_0^2\right\} r^2 \, \d r , \\
\phi_{0,i} [v_0, v_2, v_4]
     &:=  \int_1^R    \bigg\{\frac{\lambda_{0,i}}{3}|\der{v_0}|^2   + |\der{v_2}|^2   +  (\lambda_{0,i} - 2)|\der{v_4}|^2 \\
            &\qquad + \frac{1}{r^2}\bigg[\frac{\lambda_{0,i}(\lambda_{0,i} + 6)}{3} v_0^2 + (\lambda_{0,i} + 4) v_2^2  + (\lambda_{0,i} - 2)^2 v_4^2 \\
            &\qquad\qquad\qquad   - 4\lambda_{0,i}\,v_0\,v_2 + 4(\lambda_{0,i} - 2) \,v_2\,v_4 \bigg] \\
            &\qquad + \frac{\lambda_{0,i}}{3} f_0(h) v_0^2  + f_2(h) v_2^2 + (\lambda_{0,i}  - 2) f_4(h) v_4^2  \bigg\}\,r^2 \, \d r .
\end{split}
\end{equation}
The functions $v_0, \, v_2$ and $v_4$ depend on the radial
variable $r$ alone and belong to
\[
 H^1_0(1, \, R) = \left\{ w\in L^2 ([1, \, R]; \, \d r) \colon \der{w} \in L^2([1, \, R]; \, r^2\d r) \textrm{ and } w(1) = w(R) = 0 \right\} ,
\]
and $\lambda_{0,i} := i(i + 1)$. More precisely, in
\cite[Proposition~3.2]{ignat2014} and \cite[Proposition~3.4]{ignat2014}, the authors study the second variation
of the Landau-de Gennes energy about the radially symmetric solution (which has an isotropic point at the origin) on all of~$\Rr^3$ 
close to $t=0$, with the one-constant elastic energy density, $\frac{1}{2}|\grad \Qvec|^2$ and show that the second variation
is non-negative definite if and only if the functionals 
\begin{equation} \label{ghi_{0,i}}
\begin{split}
\Phi_{0,0} [v_0]
    &:= \frac{2}{3}\int_1^R \left\{|\der{v_0}|^2 + \frac{6}{r^2} v_0^2 + g_0(s)v_0^2\right\} r^2 \, \d r , \\
\Phi_{0,i} [v_0, v_2, v_4]
     &:=  \int_1^R    \bigg\{\frac{\lambda_{0,i}}{3}|\der{v_0}|^2   + |\der{v_2}|^2   +  (\lambda_{0,i} - 2)|\der{v_4}|^2 \\
            &\qquad + \frac{1}{r^2}\bigg[\frac{\lambda_{0,i}(\lambda_{0,i} + 6)}{3} v_0^2 + (\lambda_{0,i} + 4) v_2^2  + (\lambda_{0,i} - 2)^2 v_4^2 \\
            &\qquad\qquad\qquad   - 4\lambda_{0,i}\,v_0\,v_2 + 4(\lambda_{0,i} - 2) \,v_2\,v_4 \bigg] \\
            &\qquad + \frac{\lambda_{0,i}}{3} g_0(s) v_0^2  + g_2(s) v_2^2 + (\lambda_{0,i}  - 2) g_4(s) v_4^2  \bigg\}\,r^2 \, \d r .
\end{split}
\end{equation}
are non-negative. Here $v_0, v_2, v_4$ are arbitrary functions of $r$ belonging to $H^1(\Rr^3)$ and $g_0, g_2, g_4$ 
are functions of an order parameter $s$, where $s$ only depends on $r$ and $s$ vanishes at $r=0$. 
Therefore, the functions $g_0, g_2, g_4$ are different to the functions $f_0, f_2, f_4$ in Equation~\eqref{phi_{0,i}}
(in particular, the $s$ in Equation~\eqref{ghi_{0,i}} is different compared to the $h$ in Equation~\eqref{phi_{0,i}}). 
However, their method of proof only depends on the elastic energy density and hence, we can appeal to their result to reduce 
the study of the second variation of the radial-hedgehog solution on a bounded $3D$ shell to a study of the functionals in Equation~\eqref{phi_{0,i}} above.

Further, in \cite[Proposition~4.1]{ignat2014} and \cite[Lemma~4.3]{ignat2014}, the authors prove that $\Phi_{0,3} \geq \Phi_{0,2}$ 
and $\Phi_{0,i} \geq 0$ if $\Phi_{0,0}, \Phi_{0,1}\geq 0$, for all $i\geq 4$. Again, these conclusions only rely on the gradient 
contributions to the second variation and are independent of $g_0, g_2$ and $g_4$ in Equation~\eqref{ghi_{0,i}}.
Hence, we can transfer these results to our framework to conclude that $\phi_{0,i} \geq 0$ for all $i\geq 3$
provided that $\phi_{0,k}\geq 0$ for $k=0,1,2$ in Equation~\eqref{phi_{0,i}}.
Therefore, we just need to study the functionals $\phi_{0, i}$ for $i\in\{0, \, 1, \, 2\}$.
To this purpose, we cannot use the same method as in
\cite{ignat2014}, because in our case $h$ has an intermediate
minimum in~$[1, \, R]$ and $\der{h}$ is not positive everywhere.
Instead, we use the Hardy decomposition trick, i.e. we write the
variables $v_i$ as $v_i = hw_i$, where $h$ is the hedgehog profile
and is a classical solution of the differential equation
\eqref{eq:12}.

\begin{lemma} \label{lemma: hardy}
 Consider the functional
 \begin{equation} \label{phi}
  \phi[v] := \int_1^R \left\{ \alpha {\der{v}}^2 + \frac{\beta}{r^2} v^2 + \alpha \left(f(h) + \gamma\right) v^2 \right\} r^2 \, \d r,
 \end{equation}
 defined for $v\in H^1_0(1, \, R)$, where $f = f_2$ is given by~\eqref{f-alone}, $h$ is a minimizer of~\eqref{h energy}, and $\alpha, \, \beta, \, \gamma\in\Rr$ are fixed parameters.
 Then $\phi$ can be equivalently written as
 \[
  \phi[v] = \int_1^R \left\{ \alpha {\der{\left(\frac{v}{h}\right)}}^2 h^2 + \frac{\beta - 6\alpha}{r^2} v^2 + \alpha\gamma v^2 \right\} r^2 \, \d r .
 \]
\end{lemma}
\begin{proof} The proof follows from 
a Hardy-type trick and the calculations are similar to those employed in Theorem~\ref{prop:3}. 
The details are omitted here for brevity. \end{proof}

With the help of the previous lemma, we can now complete the
analysis of the second variation.

\begin{proposition} \label{prop: stability large t}
 There exists $t_* > 0$ such that the radial-hedgehog is a locally stable
 equilibrium for Problem~\eqref{pb:Landau-dG}, for all $t \geq t_*$ and $R > 1$.
\end{proposition}

\begin{proof}
By the previous discussion, it is enough to prove the positivity of $\phi_{0, i}$ defined above, for $i\in\{0, \, 1, \, 2\}$.
Throughout the proof, we fix $v_0, \, v_2, \, v_4\in C^\infty_c(1, \, R)$, and set $w_k := v_k/h$ for $k\in\{0, \, 2, \, 4\}$.
It is not restrictive to assume that the $v_k$'s are regular, because $C^\infty_c(1, \, R)$ is dense in $H^1_0(1, \, R)$.

\setcounter{step}{-1}
\begin{step}[Study of $\phi_{0,0}$]
We remark that in view of~\eqref{f}, we have
\begin{equation} \label{f0}
 f_0(h) = f(h) + t h^2 + \frac{3h_+}{2} \left(2 h^2 - h\right) .
\end{equation}
Hence, we can apply Lemma~\ref{lemma: hardy} to $\phi_{0, 0}$:
\[
\begin{split}
 \frac32 \phi_{0,0} [v_0] &= \int_1^R \left\{ |\der{w_0}|^2 + \left(t h^2 + \frac{3h_+}{2} (2 h^2 - h) \right) w_0^2 \right\} h^2 r^2 \, \d r \, .
 \end{split}
\]
By Lemma~\ref{lemma: h}, there exists $t_0 > 0$ such that $h \geq
1/2$ for $t \geq t_0$. As a consequence,
\[
 t h^2 + \frac{3h_+}{2} (2 h^2 - h) \geq \frac14 t_0 > 0
\]
and $\phi_{0, 0} [v_0] >0$ when $t \geq t_0$, with equality if and only if $v_0 = 0$.
\end{step}

\begin{step}[Study of $\phi_{0,1}$]
We recall the definition of $\phi_{0,1}$, noting that $\lambda_{0,1} = 2$:
\[
 \begin{split}
  \phi_{0,1} [v_0, v_2]
     &:=  \int_1^R \bigg\{\frac{2}{3}|\der{v_0}|^2 + |\der{v_2}|^2  + \frac{1}{r^2}\left(\frac{16}{3} v_0^2 + 6 v_2^2 
     - 8\,v_0\,v_2 \right) \\
            &\qquad + \frac{2}{3} f_0(h) v_0^2  + f_2(h) v_2^2  \bigg\}\,r^2 \, \d r .
 \end{split}
\]
With the help of~\eqref{f0}, we apply Lemma~\ref{lemma: hardy}
first to terms in $v_0$, followed by terms in $v_2$. We obtain
\[
 \begin{split}
  \phi_{0,1} [v_0, v_2]  &=  \int_1^R \bigg\{\frac{2}{3}|\der{w_0}|^2  + |\der{w_2}|^2  + \frac{1}{r^2}\left(\frac{4}{3} w_0^2 - 8\,w_0\,w_2 \right) \\
             &\qquad + \frac{2}{3} \left(t h^2 + \frac{3h_+}{2} (2 h^2 - h)\right) w_0^2 \bigg\} h^2 r^2 \, \d r .
 \end{split}
\]
By virtue of Lemma~\ref{lemma: hardy ineq}, we have
\[
 \begin{split}
  \phi_{0,1} [v_0, v_2]  &\geq  \int_1^R \left\{\frac{2}{3}|\der{w_0}|^2  + \frac{2}{3} \left(t h^2 + \frac{3h_+}{2} (2 h^2 - h)\right) w_0^2 \right\} h^2 r^2 \d r\\
             &\qquad + \int_1^R \left\{\frac{4}{3}h_{\textrm{min}}^2 w_0^2 - 8 |w_0\,w_2| + \frac{1}{4}h_{\textrm{min}}^2 w_2^2\right\}  \, \d r ,
 \end{split}
\]
where $h_{\textrm{min}} := \min_{[1, \, R]} h >0$. For $t \geq
t_0$, we have $h_{\textrm{min}}^2/4\geq 1/16$ and
\[
 - 8 |w_0\,w_2| \geq -256 w_0^2 - \frac{1}{16} w_2^2 \geq -256 w_0^2 - \frac 14 h_{\textrm{min}}^2 w_2^2 .
\]
Since $h$ converges uniformly to $1$ (see Lemma~\ref{lemma: h}), there exists some $t_1 \geq t_0$ such that, for all $t\geq t_1$ and $r \geq 1$,
\[
  \frac{2}{3} \left(t h^2 + \frac{3h_+}{2} (2 h^2 - h)\right) h_{\textrm{min}}^2 r^2 + \frac{4}{3}h_{\textrm{min}}^2 \geq 256 .
\]
We combine these inequalities to obtain $\phi_{0, 1}[v_0, \, v_2]
\geq 0$ for $t \geq t_1$ (with equality if and only if $v_0 = v_2
= 0$).
\end{step}
\begin{step}[Study of $\phi_{0,2}$]
Recall that $\phi_{0,2}$ is given by
 \[
  \begin{split}
  \phi_{0,2}[v_0, v_2, v_4]
        &:= \int_1^R \bigg\{2| \der{v_0}|^2 + |\der{v_2}|^2 + 4|\der{v_4}|^2 \\
        &\qquad + \frac{1}{r^2}\left(24 v_0^2 + 10 v_2^2 + 16 v_4^2 - 24 v_0 v_2 + 16 v_2 v_4\right) \\
        &\qquad + 2 f_0(h) v_0^2 + f_2(h) v_2^2 + 4 f_4(h)v_4^2\bigg\}r^2 \, \d r
\end{split}
\]
(set $\lambda_{0,2} = 6$ in Equation~\eqref{phi_{0,i}}). Given
that
\begin{equation} \label{f4}
 f_4(h) = f(h) + \frac{9h_+}{2} h
\end{equation}
and $f_0(h)$ is given by~\eqref{f0}, we can apply
Lemma~\ref{lemma: hardy}:
\[
 \begin{split}
  \phi_{0,2}[v_0, v_2, v_4]
        &= \int_1^R \bigg\{2| \der{w_0}|^2 + |\der{w_2}|^2 + 4|\der{w_4}|^2 \\
        &\qquad + \frac{1}{r^2}\left(12 w_0^2 + 4 w_2^2 -8 w_4^2 - 24 w_0 w_2 + 16 w_2 w_4\right) \\
        &\qquad + \left(2t h^2 + 3h_+ (2 h^2 - h)\right) w_0^2 + 18h_+h w_4^2\bigg\}h^2r^2 \, \d r .
\end{split}
\]
Clearly, we have
\begin{equation} \label{step3, 1}
 \begin{split}
 \phi_{0,2}[v_0, v_2, v_4]
        &\geq \int_1^R \bigg\{2| \der{w_0}|^2 + |\der{w_2}|^2 + 4|\der{w_4}|^2
        + \left(12 + 2t h^2 + 3h_+ (2 h^2 - h)\right) w_0^2 \\
        &\qquad + 4 w_2^2 + \left(18 h_+h_{\mathrm{min}} - 8\right)w_4^2 - 24 w_0 w_2 + 16 w_2 w_4 \bigg\} h^2 \, \d r .
\end{split}
\end{equation}
By applying the Cauchy-Schwarz inequality, we obtain the following
inequality
\[
 - 24 w_0 w_2 + 16 w_2 w_4 \geq - 72 w_0^2 - 4 w_2^2 - 32 w_4^2 .
\]
Recalling that $h \to 1$ uniformly (see Lemma~\ref{lemma: h}), it
is possible to find $t = t_2 \geq 0$ such that
\[
 12 + 2t h^2 + 3h_+ (2 h^2 - h) \geq 72 \qquad \textrm{and} \qquad 18 h_+h_{\mathrm{min}} - 8 \geq 32
\]
for $t\geq t_2$. Hence, from~\eqref{step3, 1}, we conclude that
\begin{equation*} \label{step3}
 \phi_{0, 2}[v_0, \, v_2, \, v_4] \geq 2\norm{\der{w_0}}_{L^2(1, \, R)}^2 + \norm{\der{w_2}}_{L^2(1, \, R)}^2 + 4\norm{\der{w_4}}_{L^2(1, \, R)}^2 ,
\end{equation*}
for any $t\geq t_2$. In particular, $\phi_{0, 2}[v_0, \, v_2, \,
v_4] \geq 0$ and  equality holds if and only if $v_0 = v_2 = v_4 =
0$.
\end{step}

In the previous steps, we have shown that $\phi_{0,0}$,
$\phi_{0,1}$ and $\phi_{0,2}$ are positive definite in their
arguments for $t \geq t_* := \max\{t_0, \, t_1, \, t_2 \}$. By the
results presented in~\cite{ignat2014}, this is enough to prove the
proposition.
\end{proof}

\begin{remark} \label{remark:tau}
 By Lemma~\ref{lemma: h}, the function~$h$ is bounded from below by a quantity which does not depend on~$R$ (and tends to~$1$ as~$t\to +\infty)$.
 Therefore,~$t_0$, $t_1$, $t_2$ and consequently,~$t_*$ can be chosen independently of~$R$.
\end{remark}

The same method of proof applies to the following result, which
yields an improved lower bound for the second variation.

\begin{proposition} \label{prop: stability ineq}
 Let $\alpha$, $\beta$ be two parameters such that $0 < \alpha < 1/2$, $0 < \beta < 9/2$.
 There exists a $t^* \geq 1$ (depending on $\alpha$, $\beta$) such that the inequality
 \[
  \frac 12 \delta^2 I[\Hvec] \geq \int_\Omega\left\{ \frac t3 h^2 v_0^2  + \alpha h_+ v_0^2 + \beta h_+ \left(v_3^2 + v_4^2 \right)\right\}~\d V
 \]
 holds for any $t \geq t^*$, $R > 1$ and any function $\Vvec\in W^{1, 2}_0(\Omega; \, S_0)$.
 Here the $v_i$'s denote the components of $\Vvec$ with respect to the
 basis~$\mathbf E$, $\mathbf F$, $\mathbf G$, $\mathbf X$, $\mathbf Y$.
\end{proposition}
\begin{proof}
Consider the quantity
\[
 \mathcal F[\Vvec] := \frac12 \delta^2 I[\Hvec] - \int_\Omega\left\{ \frac t3 h^2 v_0^2  + \alpha h_+ v_0^2 + \beta h_+ \left(v_3^2 + v_4^2 \right)\right\}~\d V .
\]
Using formula~\eqref{second variation} for the second variation, we obtain
\[
\begin{split}
 \mathcal F[\Vvec] = \int_\Omega \bigg\{ \frac{1}{2} \abs{\nabla \Vvec }^2 &+ \frac13 \left( f_0(h) - t h^2 - 3\alpha h_+\right) v_0^2 \\
                   &+ f_2(h) (v_1^2 + v_2^2) + \left(f_4(h) - \beta h_+\right) (v_3^2 + v_4^2) \bigg\}~\d V .
 \end{split}
\]
By virtue of~\eqref{f0} and~\eqref{f4}, we can write
\begin{equation} \label{f0 bis}
 f_0(h) - t h^2 - 3\alpha h_+ = f(h) + \frac{3h_+}{2} \left(2 h^2 - h -2\alpha \right)
\end{equation}
and
\begin{equation} \label{f4 bis}
 f_4(h) - \beta h_+ = f(h) + h_+ \left(\frac{9h}{2} - \beta\right) .
\end{equation}
Recalling that $h\to 1$ uniformly (as $t \to \infty$) by Lemma~\ref{lemma: h} and
since we have fixed $\alpha < 1/2$, $\beta < 9/2$, we deduce that
\begin{equation} \label{f0 - f0bis}
 \frac{3h_+}{2} \left(2 h^2 - h -2\alpha \right) \to +\infty, \qquad h_+ \left(\frac{9h}{2} - \beta\right) \to +\infty
\end{equation}
as $t\to +\infty$. 
We can now apply the same arguments as in Proposition~\ref{prop:
stability large t} to the functional $\mathcal F$. The proof
carries over almost word by word. Namely, at the end of each step
0--2, one uses the property~\eqref{f0 - f0bis} to absorb the
negative contributions. We conclude that there exists $t^*\geq 1$
such that
\[
 \mathcal F[\Vvec] \geq 0 \qquad \textrm{for all } \Vvec \in W^{1, 2}_0 (\Omega; \, \, S_0)
\]
for $t\geq t^*$.
\end{proof}

The final ingredient in the proof of Theorem~\ref{th:2} is Lemma~\ref{lemma: phi positive} below.

\begin{lemma} \label{lemma: phi positive}
 There exists $h_*\in (0, \, 1)$ such that if $h \geq h_*$ everywhere on $[1, \, R]$, then
 \[
\begin{split}
\varphi(\Vvec) &:= \frac 25 v_0^2 + v_3^2 + v_4^2 - \frac{\sqrt 6}{2}\tr\Vvec^3 + \frac 38 \abs{\Vvec}^4 + 5 \left(2(\Hvec\cdot\Vvec) + \abs{\Vvec}^2\right)^2 \geq 0
\end{split}
\]
for every $\Vvec\in S_0$, with equality if and only if $\Vvec = 0$.
\end{lemma}
\begin{proof} 
From Lemma~\ref{lemma: lower bound psi}, we
can assume without loss of generality that $v_2 = v_3 = 0$, so
 \[
\begin{split}
\varphi(v_0, \, v_1, \, v_4) &= \frac 25 v_0^2 + v_4^2 + \sqrt{6} \left(v_0v_4^2 - \frac 32 v_4v_1^2 - \frac12 v_0 v_1^2 - \frac 19 v_0^3 \right) \\
                             &\qquad + \frac38 \left(\frac23 v_0^2 + 2v_1^2 + 2v_4^2 \right)^2 + 5\left( \frac23 v_0^2 + 2\sqrt{\frac23}h v_0 + 2v_1^2 + 2v_4^2\right)^2 .
\end{split}
\]
As a function of $(v_0, \, v_1, \, v_4)\in\Rr^3$, $\varphi$ is smooth and bounded from below, since
\[
 \varphi(v_0, \, v_1, \, v_4) \geq \sqrt{6} \left(v_0v_4^2 - \frac 32 v_4v_1^2 - \frac12 v_0 v_1^2 - \frac 19 v_0^3 \right) + \frac38 \left(\frac23 v_0^2 + 2v_1^2 + 2v_4^2 \right)^2 \to +\infty
\]
as $\|(v_0, \, v_1, \, v_4)\| \to +\infty$. Thus, $\varphi$ has a
global minimum, which is also a critical point. We claim that $v_0
= v_1 = v_4 = 0$ is the unique critical point for $\varphi$, when
$h$ is sufficiently close to $1$. This implies, in particular,
that $v_0 = v_1 = v_4 = 0$ is a global minimum of $\varphi$ and
the lemma follows.

For the sake of simplicity, we denote the triplet $(v_0, \, v_1,
\, v_4)$ by  $(x, \, y, \, z)$.

\setcounter{step}{0}
\begin{step}[Any critical point satisfies $y = 0$]
A critical point $(x, \, y, \, z)$ is a solution of the system $\nabla \varphi = 0$, that is,
\begin{equation}\label{critical point}
  \begin{aligned}
  \sqrt 6\left(z^2 - \frac{y^2}{2} - \frac{x^2}{3}\right) + \frac{86}{9} \, x\left(x^2 + 3\, y^2 + 3\, z^2\right)
     + \left(\frac 45 + \frac{80}{3}\, h^2\right) x \hspace{1.9cm} & \\
     +\frac{40}{3} \sqrt 6 \, h\left(x^2 + y^2 + z^2\right) &= 0 \\
  \frac{\sqrt 6}{9} \, y \left(-27\, z - 9\, x + 129\sqrt 6 \, z^2 + 129\sqrt 6 \, y^2 + 43\sqrt 6 \, x^2 + 240 \, h x\right) &= 0  \\
  2 \sqrt 6 \, xz - \frac 32 \sqrt 6\,  y^2 + \frac{86}{3}\, zx^2 + 86\, zy^2 + 86\, z^3 + 2\, z + \frac{80}{3}\sqrt 6 \, h xz &= 0 . 
  \end{aligned}
\end{equation}
Let $y\neq 0$. Then,
\begin{equation} \label{y critical}
 y^2 = \frac{1}{129\sqrt 6}\left(27\, z + 9\, x - 129\sqrt 6\,  z^2 - 43\sqrt 6\,  x^2 - 240 \, h x\right) .
\end{equation}
We substitute this value of $y^2$ into Equation~\eqref{critical point}.
Note that the $xy^2$-term in the first equation expands into several terms:
\[
 \frac{86}{3} x y^2 = \sqrt 6\, xz + \frac{2}{\sqrt 6}\, x^2 - \frac{86}{3} \, xz^2 - \frac{86}{9} \,  x^3 - \frac{160}{3\sqrt 6} \, h x^2 .
\]
Thus, the cubic $x^3$ and~$xz^2$-terms cancel out when we inject this expression into~\eqref{critical point}.
Similarly, the~$xz^2$ and~$z^3$-terms in the third equation cancel out because
\[
 \frac{86}{3} z y^2 = 3\sqrt 6\, z^2 + \sqrt 6\, xz - 86 \, z^3 - \frac{86}{3} \,  x^2z - \frac{160}{\sqrt 6} \, h xz .
\]
So all the cubic terms in~\eqref{critical point} disappear and we obtain
\[
 \begin{aligned}
  \frac{329}{430} \, x + \frac{80}{43} \, hx + \frac{80}{43} \, h^2 x - \frac{9}{86} \, z + \frac{120}{43} \, hz + 
    \frac{\sqrt{6}}{6} \, x^2 + \sqrt{6} \, xz + \frac{3}{2}\sqrt{6} \, z^2 &= 0 \\
  - \frac{9}{86} \, x + \frac{120}{43} \, hx + \frac{145}{86} \, z + \frac{\sqrt{6}}{2} \, x^2 + 3 \sqrt{6} \, xz 
    + \frac{9}{2}\sqrt{6} \, z^2 &= 0
\end{aligned}
\]
This system can be further simplified by taking a linear combination of the two equations (we multiply the first equation by~$3$, the second by~$-1$ and add the two equations). We obtain
\begin{equation*} 
 \begin{aligned}
  516\, x + 600\, hx + 1\, 200\, h^2 x - 430 \, z + 1\, 800 hz &= 0 \\
  - 9\, x + 240 \, h x + 145\, z + 43\sqrt 6\,  x^2 + 258\sqrt 6\, xz +  + 387\sqrt 6\,  z^2 &= 0.
 \end{aligned}
\end{equation*}
This is a system of second degree in $(x, \, z)$, so it can be easily solved. 
There are two solutions: $x = z = 0$, and~$x = x_0(h)$, $z= z_0(h)$ where~$x_0$, $z_0$ are algebraic functions of~$h$:
\[
 \begin{aligned}
  x_0(h) &:= -\frac{125\sqrt 6}{2} \cdot \frac{473 - 604 \, h - 7\,480 \, h^2 + 7\,200 \, h^3}{978\,121 + 3\,560\,400 \, h^2 + 3\,240\,000 \, h^4} \\[6pt]
  z_0(h) &:= \frac{75\sqrt 6\left(43 + 50 \, h + 100 \, h^2\right)\left(40 \, h^2 - 32 \, h-11\right)}{978\,121 + 3\,560\,400 \, h^2 + 3\,240\,000 \, h^4} .
 \end{aligned}
\]
By substituting~$x = x_0(h)$, $z = z_0(h)$ into Formula~\eqref{y critical}, we write $y^2$ as an algebraic function of~$h$.
Taking the limit as $h\to 1$, we get 
\[
 y^2 \to - \frac{441\,133\,354\,650}{60\,505\,388\,947\,441} < 0 
\]
which is clearly a contradiction. Thus, there exists a value $h_0\in (0, \, 1)$ such that any critical point of $\varphi$ satisfies $y = 0$ for $h \geq h_0$.
\end{step}

\begin{step}[Any critical point satisfies $z=0$]
 We set $y=0$ in Equation~\eqref{critical point}:
 \begin{equation} \label{xz critical 2}
  \begin{aligned}
   \sqrt 6\, z^2 - \frac{\sqrt 6}{3}\, x^2 + \frac{86}{9}\,x\left(x^2 + 3\, z^2 \right) + \left( \frac 45 + \frac{80}{3}\, h^2\right)\, x + \frac{40}{3}\sqrt 6 \, h\left(x^2 + z^2\right) &= 0 \\
   \frac{\sqrt 6}{9} z \left(3\sqrt 6 + 18 \, x + 129\sqrt 6\, z^2 + 43\sqrt 6\, x^2 + 240 \, hx \right) &= 0
  \end{aligned}
 \end{equation}
Suppose that $z\neq 0$. Then,
\[
 z^2 = -\frac{1}{129}\left(3 + 3\sqrt 6\, x + 43\, x^2 + 40\sqrt 6 \, hx\right) .
\]
We eliminate the variable $z$ from~\eqref{xz critical 2} and
obtain an equation for $x$:
\[
 -860\sqrt 6\, x^2 -4\left(1 - 600\, h + 300 \, h^2\right)\, x - 15\sqrt 6 - 200 \sqrt 6\, h  = 0 .
\]
This equation has no real root for $h=1$. Therefore, we conclude
that there exists $h_*\in (h_0, \, 1)$ such that any critical
point of $\varphi$ has $z = 0$ for $h \geq h_*$.
\end{step}

\begin{step}[Conclusion]
 Substituting $y = z = 0$ into Equation~\eqref{critical point} results in an equation for $x$:
 \[
 \frac{1}{45}x \left(430\, x^2 + \left(-15\sqrt 6 + 600 \sqrt 6\, h\right)\, x + 36 + 1 \, 200 h^2 \right) = 0 .
 \]
 The discriminant of the second-order factor is
 \[
  \left(-15\sqrt 6 + 600 \sqrt 6\, h \right)^2 - 4\cdot 430 \left(36 + 1 \, 200 \, h^2\right) = -60 \, 570 - 108 \, 000 \, h + 96 \, 000 \, h^2 ,
 \]
 which is strictly negative for $0 \leq h \leq 1$.
 Thus, the system~\eqref{critical point} has the unique solution $x = y = z = 0$ for $h \geq h_*$. \qed
\end{step}
 \let\qed\relax
\end{proof}

The proof of Theorem~\ref{th:2} now follows.

\begin{proof}[Proof of Theorem~\ref{th:2}]
Fix a radius~$R > 1$ and let~$h_*$ be given by Lem\-ma~\ref{lemma: phi positive}.
From Lemma~\ref{lemma: h}, we can find $\tau_1$ such that when $t\geq \tau_1$, 
the inequality $h \geq h_*$ holds for $r \in (1, \, R)$. Let $\tau_2$ be such that
\begin{equation} \label{t h+}
\frac{t}{8} \geq \frac{43}{8}h_+
\end{equation}
for $t \geq \tau_2$ (such a $\tau_2$ exists because $h_+\leq C\sqrt t$ for $t\gg 1$).
Choose $\alpha = 2/5$, $\beta = 1$ and let $t^* = t^*(2/5, \, 1)$ be given by Proposition~\ref{prop: stability ineq}. Finally, set
\[
 \tau := \max\{\tau_1, \, \tau_2, \, t^*\} .
\]
Note that~$\tau_1$,~$\tau_2$, $t^*$ and hence~$\tau$ can be chosen independently of~$R$, by Remark~\ref{remark:tau}. 
We fix $t\geq \tau$ and an admissible map $\Qvec\in W^{1, 2}(\Omega; \, S_0)$, and we write $\Qvec = \Hvec + \Vvec$.
From Equation~\eqref{second variation V}, we deduce that
 \[
  \begin{split}
   I[\Qvec] - I[\Hvec] &= \frac 12 \delta^2 I[\Hvec] + \int_\Omega \left\{-\frac{\sqrt{6}h_+}{2} \tr\Vvec^3 + \frac{t + 3h_+}{8} \left(4(\Hvec\cdot\Vvec)\abs{\Vvec}^2 + \abs{\Vvec}^4\right)  \right\}~\d V .
  \end{split}
 \]
 Using Proposition~\ref{prop: stability ineq} with $\alpha = 2/5$ and $\beta = 1$, we obtain
 \begin{equation} \label{remark1}
  \begin{split}
   I[\Qvec] - I[\Hvec] &\geq \int_\Omega \bigg\{\frac{2h_+}{5} v_0^2 + h_+\left(v_3^2 + v_4^2 \right) -\frac{\sqrt{6}h_+}{2} \tr\Vvec^3  \\
   &\qquad + \frac{3h_+}{2}(\Hvec\cdot\Vvec)\abs{\Vvec}^2 + \frac{3h_+}{8}\abs{\Vvec}^4 + \frac{t}{8}\left(2(\Hvec\cdot\Vvec) + \abs{\Vvec}^2\right)^2 \bigg\}~\d V .
  \end{split}
 \end{equation}
 Clearly, it holds that
 \begin{equation} \label{remark2}
  - \frac{3h_+}{2}(\Hvec\cdot\Vvec) \abs{\Vvec}^2 \geq - \frac{3h_+}{8}\left(2(\Hvec\cdot\Vvec) + \abs{\Vvec}^2\right)^2 .
 \end{equation}
 Combining~\eqref{remark1},~\eqref{remark2} and~\eqref{t h+}, we
 deduce that
 \[
 I[\Qvec] - I[\Hvec] \geq h_+ \int_\Omega \varphi(\Vvec)~\d V ,
 \]
where $\varphi$ is the function defined in Lemma~\ref{lemma: phi positive}.
Since $\varphi\geq 0$ and $\varphi(\Vvec) = 0$ if and only if $\Vvec = 0$, the theorem follows.
\end{proof}

\section{Conclusions}
\label{sec:con}

We study nematic equilibria within 3D spherical shells with Dirichlet radial conditions on both spherical concentric boundaries.
We work within the Landau-de Gennes theory for nematic liquid crystals and show that this problem has a radial equilibrium, 
which we refer to as the radial-hedgehog solution. We define the radial-hedgehog solution by analogy with the definition
on a 3D spherical droplet, as used in~\cite{mkaddem&gartland2, sima2012} and work with temperatures below
the critical nematic supercooling temperature, defined by $t\geq 0$. 
We focus on the local and global stability of this defect-free equilibrium.
We prove that the radial-hedgehog solution is the unique global
minimizer for all shell widths smaller than an explicit critical value computed in Theorem~\ref{th:1}. 
In Theorem~\ref{th:2}, we prove that the radial-hedgehog solution is the unique global minimizer for all temperatures~$t\geq \tau$,
for all shell widths (independent of $t$) and we specify the largeness of $\tau$ in Propositions~\ref{prop: stability large t}
and~\ref{prop: stability ineq} and Lemma~\ref{lemma: h}.
For both theorems, we control the second variation and establish conditions for local stability 
i.e. positivity of the second variation, by means of a Poincar\'e-type inequality in Theorem~\ref{th:1} and by adapting arguments
from~\cite{ignat2014} in Theorem~\ref{th:2}. For Theorem~\ref{th:2}, we obtain an explicit positive lower bound for the second
variation, as stated in Proposition~\ref{prop: stability ineq}. The final steps focus on how to control the residual terms in the
energy expansions (which include quadratic residual terms from the second variation, cubic and quartic terms in the energy expansion) 
using the positive bounds for the second variation. Our methods will work or can be adapted to study the local and global stability of
defect-free uniaxial equilibria in model geometries, where the order parameter is bounded away from zero. One physically relevant 
example of such a defect-free equilibrium is the Vertically Aligned State in the Zenithally Nematic Device (see~\cite{alexraisch}).
Equally, these methods do give insight into how one may rigorously study the loss of stability of defect-free equilibria which might
indicate the creation of defects or singularities. 
It is possible that defect-free equilibria lose global minimality / stability before they become unstable. 
The computations in Sections~\ref{sec:global} and~\ref{sec:t} could be useful in tracking the global stability of defect-free
equilibria as a function of model parameters and onset of competing equilibria with lower energy. 
This will be investigated in future work.

\section*{Acknowledgements}

A.M.'s research is 
supported by an EPSRC Career Acceleration Fellowship EP/J00\-1686/1 and EP/J001686/2, an OCIAM
Visiting Fellowship, Royal Society International Exchange Grant and a University of Bath internationalization grant. This work started when A.M. visited M.R. in April 2013 and her visit was funded by a Royal Society International Exchange Grant.
M.R. visited A.M. in November 2013 and A.M. and M.R. gratefully acknowledge funding from the Royal Society International Exchange Grant.
G.C. would like to thank the Department for Mathematical Sciences, Bath, for hosting him and funding two collaborative visits whilst this work was being carried out.
He is also very grateful to his Ph.D. advisor Fabrice Bethuel, for constant support and helpful advice.




 \bibliographystyle{acm} 
 \bibliography{PhysicaD2015}
 
\begin{figure} \label{fig: G}
 \includegraphics[width = .9\textwidth, height = .27\textheight]{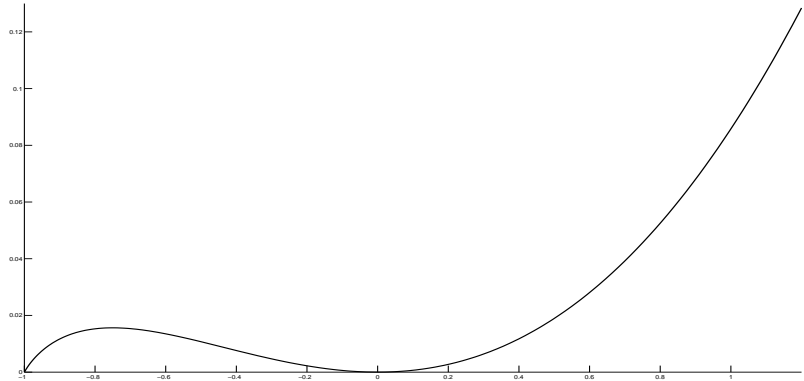}
 \caption{A plot of of the function $G$.}
\end{figure}


%
%

\end{document}